\documentclass[12pt,a4paper]{article}

 \usepackage[a4paper]{geometry}
\usepackage{hyperref}
\usepackage{amsmath,amssymb,mathtools,amsthm,amsopn}

\usepackage{fancyhdr}
\usepackage[dayofweek]{datetime}

\usepackage{cite}
\usepackage{lettrine}
\usepackage{mdwlist}
\usepackage{tocloft}
\usepackage{mathrsfs}
\renewcommand{\phi}{\varphi}

\usepackage{fourier-orns}

\newcommand{\I}{\mathcal{I}}
\newcommand{\U}{\mathcal{U}}

\newcommand{\abs}[1]{\lvert#1\rvert}

\newcommand{\g}{\gamma}

\newcommand{\e}{\varepsilon}
\renewcommand{\r}{\varrho}
\renewcommand{\rho}{\varrho}

\newcommand{\s}{\sigma}

\newcommand{\n}{\nu}

\renewcommand{\cal}[1]{\mathcal{#1}}

\newcommand{\wt}{\widetilde}

\renewcommand{\d}{\delta}
\newcommand{\Eucl}{\textup{Euc}}
 
\newcommand{\p}{\partial}

\newcommand{\R}{\mathbb{R}}
\newcommand{\N}{\mathbb{N}}

\newtheoremstyle{pippo}  
  {}       
  {}       
{\slshape} 
 {}        
  {\bfseries}  
  {.}   
  {1ex}       
  {}           

\newtheorem{theorem}{Theorem}[section]
\newtheorem{lemma}[theorem]{Lemma}

\newtheorem{corollary}[theorem]{Corollary}

\newtheorem{example}[theorem]{Example}

\theoremstyle{remark}
\newtheorem{remark}[theorem]{Remark}

\renewcommand{\d}{\delta}
\renewcommand{\t}{\tau}

\renewcommand{\O}{\mathcal{O}}
\renewcommand{\P}{\mathcal{P}}

\renewcommand{\a}{\alpha}
\renewcommand{\b}{\beta}
\newcommand{\loc}{\textup{loc}}
\DeclareMathOperator{\Lip}{Lip}
\DeclareMathOperator{\Span}{span}

\newcommand{\norm}[1]{\left\Vert#1\right\Vert}
\numberwithin{equation}{section}
\allowdisplaybreaks[4]
\let\oldbibliography\thebibliography
\renewcommand{\thebibliography}[1]{\oldbibliography{#1}\setlength{\itemsep}{0pt}}

\begin{document}

 \title{A Frobenius-type theorem  for singular
Lipschitz~distributions}
\author{Annamaria Montanari and Daniele Morbidelli}

\date{\today}
\maketitle



\begin{abstract}
We prove a Frobenius-type theorem for singular distributions generated by a
family of locally Lipschitz continuous vector fields satisfying almost everywhere a 
quantitative finite type condition.
\end{abstract}

\section{Introduction and main result}
Consider a family  $\P:=\{Y_1, \dots, Y_q\}$ 
of vector fields in $\R^n$.
For all  $x\in\R^n$  let $P_x:= \Span\{Y_{1,x},\dots, Y_{q,x}\},$
where $Y_{j,x}:= Y_j(x)$ denotes the vector field $Y_j$ evaluated at
$x\in\R^n$.

As a consequence of the  classical Frobenius theorem, it is known that if  the vector fields
are smooth, the rank   $p_x:=\dim P_x$ is constant in $\R^n$ and  if the family
satisfies the
\emph{involutivity condition}
\begin{equation}\label{giacozzo} 
 [Y_i, Y_j](x)\in P_x\quad\text{for all $x\in \R^n$
$i,j\in\{1,\dots,q\}$,} 
\end{equation}
then  for all $x\in\R^n$ there is a   smooth immersed submanifold $M^x$ containing $x$ and 
 with $T_y M^x =
P_y$ for all $y\in M^x$; see \cite{Chevalley}. The result still holds if
one removes the constant-rank assumption, but the involutivity assumption
\eqref{giacozzo} does not suffice. Hermann \cite{Hermann} has shown that
a sufficient
condition for smooth distributions is the \emph{finite-type  involutivity condition}
\begin{equation}\label{ft} 
      [Y_i, Y_j] = \sum_{1\le k\le q} c_{ij}^k Y_k \quad\text{for all $i,j\in\{1,\dots,q\}$},
\end{equation}
where the coefficients $c_{ij}^k$ are locally bounded.

In modern terms, a ``maximal'' choice the integral manifold $M^x$ can be identified as the following 
\emph{orbit} or \emph{leaf}:   
 \begin{equation}\label{orbs} \begin{aligned}
 \cal{O}^{x}:= \{y: d(x,y) <\infty\},         
\end{aligned}\end{equation}
where $d$ denotes the Carnot--Carath\'eodory distance
\begin{equation}\label{dosto} 
\begin{aligned}
      d(x,y):= \inf\Bigl\{T\geq0 : 
\text{ $\exists \; \gamma$ subunit on $[0,T]$ with $\gamma(0)=x$, $\gamma(T)=y$\big\}.}  
\end{aligned}
\end{equation}
Recall that a  path $\gamma:[0,T]\to \R^n$ is subunit if it is  Lipschitz and satisfy 
for a.e.~$t$ the ODE $\dot\gamma(t)= \sum_{j=1}^q u_j(t) Y_j(\gamma(t))$ with $u\in L^\infty
((0,T),\R^q)$ and $\abs{u(t)}\le 1$ for a.e.~$t$. We agree that $d(x,y)=\infty$ if there are no subunit paths
connecting $x$ and~$y$.

Generalizations of the Frobenius theorem hold for nonsmooth vector
fields. Hartman \cite{Hartman}, Simic
\cite{Simic96} and Rampazzo
\cite{Rampazzo}  have shown that,  given a
constant-rank family of locally Lipschitz vector fields which satisfies
\eqref{giacozzo} almost everywhere  or as a set-valued commutator
\cite{RamSus}, then 
integral manifolds  are  $C^{1,1}$ smooth.

In applications it is sometimes necessary a Frobenius theorem for singular
distributions, in which the dimension of the 
subspace $P_x$
may vary with $x.$
This happens for example when one considers the singular distribution associated with the Hamiltonian vector fields  
on a Poisson manifold.  For a good account on this
application we refer to \cite [Section 3.4]{Adler} and 
\cite[Appendix 3]{Libermann}.

In
this note, under an assumption similar to \eqref{ft},  we show 
that a  regularity result for orbits of Lipschitz vector fields
holds in the nonconstant rank case. Here is our main result. Below $\t_d$ denotes the topology associated with the control distance $d$ defined in \eqref{dosto}. Let also for $x\in\R^n$ and $r>0$,
$P_x^r: =
\sum_{j=1}^q c_j Y_{j,x}: \abs{c}\le r\}$.

\begin{theorem}\label{manno} 
 Let
 $\P:=\{Y_1,\dots, Y_q\} $ be a family of locally Lipschitz continuous vector fields in
$\R^n$. Write $Y_j=: g_j\cdot\nabla$ and assume that for any bounded
open
set
$\Omega\subset\R^n$ there is $C= C_{\Omega} >0$ such that
\begin{equation}\label{tostapane}
[Y_{j}, Y_k]_x :=  (Y_j g_k (x)- Y_k g_j(x))\cdot\nabla
\in P_x^{C_\Omega}\quad \text{for   a.e. $x\in\Omega$}.
\end{equation}
Then for all $x\in \R^n$ the orbit $\O^x$ with topology $\t_d$ has a structure of  connected 
  $C^{1,1}$
immersed submanifold of $\R^n$ with $T_y \O^x = P_y$ for all $y\in
\O^x$. 
\end{theorem}
\begin{remark}
 Concerning the statement, note that
assumption \eqref{tostapane} is meaningful in view of the
Rademacher theorem. Moreover note the following facts:
\begin{itemize*}

\item [(a)] Our orbits are defined in a  different way from the usual Sussmann's orbits. Indeed, the Sussmann's orbit associated with $\P$ is the set of points in $\R^n$ which are reachable from $x$ via a piecewise integral curve of vector fields in $\P$. It is trivially
 $\O^x_{\textup{Sussmann}}\subset\O^x$.
 It is known that  $\O^x_{\textup{Sussmann}}=\O^x$ in more regular situations.
However, it is not known to the authors  whether  these sets agree if the vector fields are  only Lipschitz continuous.

\item [(b)] If one removes the quantitative assumption  \eqref{tostapane}, the statement of  Theorem
\ref{manno} fails. 
This can be seen in the well known example $\P = \{\p_1,
\exp(-1/x_1^2)\p_2\}$,
where $\O^x = \R^2$ for all choices of $x$. Therefore, equality
$T_y \O^x = P_y$ fails at any $y$ in the $x_2$-axis.  Note that 
in such example \eqref{giacozzo} holds, but 
\eqref{tostapane} is not satisfied. 

\item [(c)] One could ask whether or not 
any orbit $\O$  of a finite family $\P$ of locally Lipschitz vector
fields is a submanifold (without requiring that $T_y \O^x = P_y$ for
all $y\in\O$).
 In the smooth case this is true by Sussmann's theorem \cite{Sussmann}.
We do not know
whether a version of Sussmann's theorem for Lipschitz
vector fields holds.

\item [(d)] In the constant rank case, orbits are $C^{1,1}$ leaves  
of  a foliation, \cite{Rampazzo,HillTaylor}. \footnote{The paper \cite{HillTaylor} contains some interesting applications of the  Frobenius theorem to foliation properties for rough Levi flat CR manifolds.} In our 
  case this is in general  not possible, because the dimension of different orbits
can be different. See Example \ref{exam}.

\item[(e)] Given a family $\P = \{ Y_1, \dots, Y_q\}$ of vector fields, 
condition \eqref{tostapane} is not necessary to have the conclusions of Theorem \ref{manno}. See Example
\ref{balan} below.

\item[(f)] 
In the constant rank case, Rampazzo \cite{Rampazzo} 
proves the following statement.
 Let $P= \cup_{x\in \R^n} P_x$ be a Lipschitz distribution.\footnote{I.e.~(compare \cite{Sussmann08}), 
$P_x$ is a subspace of $T_x\R^n$
for each $x$ and moreover for any $x\in \R^n$ and for each $v\in P_x$ there is a neighborhood 
$U$ of $x$ in $\R^n$ and $X\in \Lip(U,\R^n)$ such that $X(x) = v$ and $X(y)\in P_y$ for all $y\in U$.}
Let $\Omega$ be an open set and let $X,Y\in\Lip(\Omega, \R^n)$ be vector fields tangent to $P$ at any point 
of $\Omega$. Then, for almost all $x\in \Omega$ we have $[X, Y](x)\in P_x$. 
We do not have a proof of such statement for singular distributions. 
\end{itemize*}

\end{remark}

\begin{example}\label{exam} 
      Let $F(x_1):= x_1\abs{x_1}$ and $f(x_1):= F'(x_1)= 2\abs{x_1}$.
Define $Y_1:= \p_1+f(x_1)\p_2$ and $Y_2:=\abs{x_2-F(x_1)}\p_2$.
The vector fields $Y_1,Y_2$ are Lipschitz and satisfy \eqref{tostapane} because
$[Y_1, Y_2]=0$ a.e. We have
$\O^{(0,0)}=\{x: x_2= F(x_1)\}$ is a one dimensional $C^{1,1}$
graph. For any $\wt x\in \R^2$ with $\wt x_2>F(\wt x_1)$ we have
$\O^{\wt x}= \{x: x_2>F(x_1)\}$. Finally, if $\wt x_2<F(\wt x_1)$, we have
$\O^{\wt x} = \{x:x_2<F(x_1)\}$.\end{example}

The following example shows that, 
already in the setting of smooth vector fields, the finite type involutivity condition \eqref{tostapane} is not a property of the distribution, but it concerns the spanning family.

\begin{example}\label{balan} \footnote{See the unpublished version of \cite{Balan} available at 
\url{http://www2.math.umd.edu/~rvbalan/PAPERS/MyPapers/distrib.pdf}.}
      Let $x= (x_1, x_2)\in \R^2$. Let $\P= \{Y_1, Y_2\}$ where 
\[
Y_1= e^{-1/\abs{x}^2} \p_1\quad\text{and}\quad Y_2= \abs{x}^2\p_2 .
\]
We have $\O^0= \{0\}$ and $\O^x = \R^2\setminus\{0\}$ for all $x\in \R^2\setminus\{0\}$. Note that $T_y(\R^2\setminus\{0\})
= P_y$ for all $y\in \R^2\setminus\{0\}$. On the other side, observe that we can write uniquely 
for all $x\neq 0$, 
\[
[Y_1, Y_2] = \frac{2x_2}{\abs{x}^2}Y_1 
+ \frac{2x_1 e^{-1/\abs{x}^2}}{\abs{x}^2}Y_2.
\]
Since the function $x\mapsto  \frac{2x_2}{\abs{x}^2}$ is unbounded in any neighborhood of the origin, we conclude that 
\eqref{tostapane} does not hold. Observe finally that this phenomenon does not occur if we change the family with the 
(analytic) family  $\{ \abs{x}^2\p_1, \abs{x}^2\p_2\}$.
\end{example}

Before closing this introduction we briefly describe the proof. 
As already appeared in the paper \cite{Hermann}, the main step is to show that 
the rank  $p_y$ is constant as $y$ belongs to a fixed orbit $ 
\O^x$. This  can be done in the smooth 
case by rectifying one of the vector fields and by an ODE argument (see \cite{Hermann}).
This procedure is not available under our low
 regularity
 assumptions. However we are able to show constancy of the rank along orbits by an approximation argument 
which involves Euclidean mollification of the original vector fields and differentiation of suitable wedge products 
along a given flow. 
This argument is both  a curvilinear and  nonsmooth version of the original one in \cite{Hermann} and relies on some differential formulas
  first derived  in \cite{NagelSteinWainger} and improved in \cite{Street} and \cite{MontanariMorbidelli11d}.
 This is achieved in subsection \ref{dutt}.

After establishing  the constancy of the rank along a given orbit $\O$, 
we will  need to construct local $C^{1,1}$ coordinates on $\O$. 
Here the classical  idea is to construct 
new vector fields which span the same distribution, but whose flows commute (see \cite{Hermann77}).\footnote{It seems that this commutativity argument  appears in the
original Clebsch's proof, which is prior to Frobenius' one, see the historical
paper \cite[Theorem 5.3]{Hawkins}.}
This can be done in the constant rank case in a  full Euclidean neighborhood of any fixed point. 
In the singular case, this construction can be done only in a $d$-neighborhood of a given point
 and it is not clear whether or not it can be extended in any Euclidean neighborhood. Therefore, 
since we are working with 
locally Lipschitz vector fields, 
the  commutativity  argument of \cite[Theorem 5.3]{RamSus}
does not work. We need again to work with the smooth approximations of the vector fields. 
This involves a careful  analysis of how the integrability condition \eqref{tostapane} behaves under mollifications; 
see subsection~\ref{datt}.

\section{Proof of the main result} \label{frobo}
\paragraph{Notation.} Our notation here are similar to
\cite{MontanariMorbidelli11d,MontanariMorbidelli11c}.
 Let $\P:=\{Y_1,\dots, Y_q\}$ be a family of
locally
Lipschitz continuous vector fields. Write $Y_j = : g_j\cdot\nabla$ and 
define for any $p,\mu\in \N$, with $1\le p\le \mu$,
\[
\I(p,\mu)  := \{I=(i_1, \dots, i_p): 1\le i_1<i_2< \cdots<i_p\le \mu\} .      
\]
  Define
 $ p_x:= \dim  \Span   \{ Y_{j,x} : 1\le j\le q\}.$
Obviously, $p_x\le \min\{n,q\}$. Then for any $p\in
\{1,\dots,\min\{n,q\}\}$,
let
\begin{equation*}
 Y_I(x):=Y_{I,x} : = Y_{i_1,x}\wedge\cdots\wedge Y_{i_p,x}\in
{\textstyle\bigwedge}_p
T_x\R^n\sim {\textstyle\bigwedge}_p\R^n \quad\text{for all $I\in
\I(p,q)$}
\end{equation*}
and, for all $K\in \I(p,n)$ and $I\in
\I(p,q)$
\begin{equation*}
Y_I^K(x): = dx^K(Y_{i_1}, \dots, Y_{i_p}) (x)
 : = \det
\left[\begin{smallmatrix}
        g_{i_1}^{k_1}(x) &\cdots &g_{i_p}^{k_1}(x)
  \\  \vdots& \vdots&\vdots
\\ g_{i_1}^{k_p}(x) &\cdots &g_{i_p}^{k_p}(x)
     \end{smallmatrix}\right]
.
\end{equation*}
Here we let $dx^K:=dx^{k_1}\wedge \cdots \wedge d x^{k_p}$  for any
$K=(k_1,\dots, k_p)\in
\I(p,n)$.

Let $e_1, \dots, e_n$ be the canonical basis of $\R^n$. The family
$e_K:= e_{k_1}\wedge\cdots\wedge e_{k_p}$, where 
$K\in\I(p,n)$,
gives an orthonormal basis of $\bigwedge_p\R^n$.
 Then we  have the orthogonal
decomposition
$
 Y_I(x)  =\sum_{K}Y_J^K(x) e_K\in {\bigwedge}_p \R^n
$.

Consider
the linear system $\sum_{k=1}^p\xi^k
Y_{i_k}= W$ for some
$W\in\Span\{Y_{i_1}, \dots, Y_{i_p}\}$. If   $\abs{Y_I}\neq 0$, then the Cramer's
rule  gives the unique solution
\begin{equation}
\label{cromo} 
\xi^k = \frac{\langle Y_I, \iota^k(W)
Y_I\rangle}{\abs{Y_I}^2}\quad\text{for each $k=1,\dots, p$,}
\end{equation}
where we let $
\iota^k(W)
Y_I:= Y_{(i_1,\dots, i_{k-1})}\wedge
 W\wedge Y_{(i_{k+1},\dots,i_p)}$.

Finally, for $p\in\{1,\dots,\min\{n,q\}\}$, introduce the
vector-valued function
\begin{equation*}
\begin{aligned}
\Lambda_p(x)& := \bigl(  Y_J^K (x)
\bigr)_{{J\in \I(p, q)},
K\in
\I(p, n)}.
\end{aligned}
\end{equation*}
Note that $\abs{\Lambda_p(x)}^2 =
\sum_{I\in \I(p, q)}
\abs{Y_I(x)}^2= \sum_{I\in \I(p, q),K\in\I(p,n)}
Y_I^K(x)^2$.

\subsection{Invariance of the dimension \texorpdfstring{$p_x$}{px} on the orbit} \label{dutt} 
Write $Y_j =
:g_j\cdot\nabla$.
For each $x\in\R^n$, let $P_x:= \Span\{Y_{1,x},\dots, Y_{q,x}\}$ and 
 for all $r>0$ 
 let $ P_x^r:=
\bigl\{\sum_{1\le j\le q} c_j Y_{j,x}: \abs{c}\le r\bigr\}$. Finally
define
$p_x:=\dim\Span\{Y_{j,x}\}$.
Let $\O^x$
be the orbit of the family $\cal{P}$
containing $x$, see \eqref{orbs}. Equip $\cal{O}$ with the topology~$\t_d$.
\begin{theorem}
\label{testo} Let
 $Y_j\in \Lip_{\loc}(\R^n)$ and assume that for any bounded open
set
$\Omega\subset\R^n$ there is $C= C_{\Omega} >0$ such that
\begin{equation}\label{tostapanino}
[Y_{j}, Y_k]_x :=  (Y_j g_k(x) - Y_k g_j(x))\cdot\nabla
\in P_x^{C_\Omega}\quad \text{for   a.e. $x\in\Omega$}.
\end{equation}
Then, for all $x_0\in\R^n$, we have
$p_x= p_{x_0} 
$ for all $x\in \O_\P^{x_0}$.
\end{theorem}
Concerning formula \eqref{tostapanino}, note that $Y_j g_k(x)= g_j(x)\cdot\nabla g_k(x) 
$ exists for almost all
$x\in \R^n$
by the Rademacher theorem.

In order to prove Theorem \ref{testo} we start with  a measurability lemma.
\begin{lemma}\label{lomi}  Assume that \eqref{tostapane} holds. Then 
  there are measurable functions $c_{jk}^i:\R^n\to \R$ such
that 
\begin{equation}\label{mistola}
 [Y_j, Y_k]_x =\sum_{1\le i\le q} c_{jk}^i (x) Y_{i,x}  \quad\text{for
almost all
$x\in \R^n$ and all $j,k= 1,\dots, q$,}
\end{equation}
where $\|c_{jk} \|_{L^\infty(\Omega)}\le C_{\Omega}$  for
all
$j,k\in\{1,\dots,q\}$.
\end{lemma}
\begin{proof} For all $x\in \R^n$ let $Y_x:= [Y_{1, x},\dots,
Y_{q,x}]\in \R^{n\times q}$.  
Let $Y_{x}^\dag$ be
its Moore--Penrose inverse.
Therefore, at any differentiability point $x$ of both $g_j$ and $g_k$,
the vector 
$c_{jk}(x):=Y_x^\dag (Y_jg_k(x)- Y_kg_j(x))\in \R^q$ is the
least-norm solution of the system $Y_x
\xi = Y_jg_k(x) - Y_k g_j(x)$, with $\xi\in \R^q$.
Therefore  $\abs{c_{jk}(x)}\le C_\Omega$ for a.e.~$x\in \Omega$.
Measurability
follows from the approximation formula 
\begin{equation*}
Y_x^\dag (Y_jg_k - Y_k g_j)(x)= \lim_{\delta\to 0}(\delta I
+ Y^T_x
Y_x )^{-1} Y_x^T
(Y_jg_k - Y_k g_j)(x). 
\end{equation*}
(see the appendix of \cite{MontanariMorbidelli11a}).  
The proof is concluded.
\end{proof}

Recall now some properties of mollifiers from \cite{RamSus}.
For any    $f\in L^1_\loc(\R^n)$, let  $f^{(\s)}(x) :=   \int_{\R^n}
f(x-\s y) \chi(y) dy$, where $\chi\ge 0$ is a smooth averaging kernel
supported in the unit ball. Let also $Y^{(\s)}_j : = g_j^{(\s)}\cdot
\nabla$
be the smooth approximation of the vector field $Y_j$. Denote for all
$p\in
\{1,\dots,q\}$ and for all  $J\in\I(p, q)$,
$Y_J^\s:=
Y_{j_1}^\s\wedge\cdots\wedge Y_{j_p}^\s$. Note that    in general
$Y_J^\s\neq
Y_J^{(\s)}$, unless $J\in\I(1,q)$. 

Let now
$\Omega_0\subset\subset\Omega_1 $ be bounded open sets contained in
$\R^n$.
Then, since the vector fields are locally Lipschitz,  there are
$\wt\s=\wt\s(\Omega_0,\Omega_1)$ and $C>0$ such that  $\sup_{\Omega_0}
\abs{Y_J^\s\to Y_J}\le C\sigma$, for all $\sigma<\wt\s$ for some $C$
depending on $\| Y_j\|_{C^{0,1}(\Omega_1)}$.

Recall  the following Friedrichs-type commutator estimate (see \cite[Lemma~4.5]{RamSus}).
For all 
$\s\le \wt\s $, 
we can write on $\Omega_0$
\[
 [Y_j^{(\s)}, Y_k^{(\s)}] = [Y_j, Y_k]^{(\s)} + \s b_{jk}^\s\cdot\nabla,
\]
where the smooth functions $b_{jk}^\s$ satisfy $\sup_{\Omega_0}
\abs{b_{jk}^\s}\le
C$ for some
$C$ depending on $\norm{Y_j}_{{C^{0,1}(\Omega_1)}}$, 
$\norm{Y_k}_{{C^{0,1}(\Omega_1)}}$ and for all
$\s\le \wt \sigma$. Then, using
\eqref{mistola}, we get, for possibly
different functions $b^\s_{jk}$, also depending on the constant
$C_{\Omega_1}$ in the assumptions of Theorem~\ref{testo}
\begin{equation}\label{ingrao}
 [Y_j^{(\s)}, Y_k^{(\s)}] = \sum_{1\le i\le q} (c_{jk}^i)^{(\s)} Y_i^\s
+ \s
b_{jk}^\s\cdot\nabla.
\end{equation}
Here $c_{jk}^i$ are the measurable functions appearing in Lemma
\ref{lomi} and the functions $b_{jk}^\s$ are smooth and uniformly
bounded as $\s\to 0$.

\begin{theorem}\label{toros} 
Let $\Omega\subset\R^n$ be a  bounded set.  Then there is $C>0$
such
that for all $x\in \Omega$,  for all  subunit 
path $\gamma$ with $\gamma(0)=x$ and for all
$p\in\{1,\dots,\min\{q, n\}\}$, we have
\begin{equation}
\label{stoic2}
 \bigl| \Lambda_p(\gamma(t))- 
\Lambda_p(x)\bigr|\le
\abs{\Lambda_p(x)}(e^{C t}
-
1)\quad\text{for all $t\in [0, C^{-1}]$}.
\end{equation}
In particular,  for each $p\in\{1,\dots, q\}$ and  $I\in\I(p,q)$, we have
\begin{equation}
      \label{ippi}
      \abs{Y_I(\gamma(t))-Y_I(x)}
\le C\abs{\Lambda_p(x)}(e^{Ct}-1)\quad\text{for all $t\in[0, C^{-1}]$.}
\end{equation} 
\end{theorem}
\begin{remark}\label{pku} 
Let $x\in\Omega$,
  $p= p_x$, $I\in\I(p_x, q)$ and $\eta\in(0,1)$ be such that
\[
      \abs{Y_I(x)}>\eta\max_{K\in\I(p_x,q)}\abs{Y_K(x)}.
\]
Then
\begin{equation}\label{added} 
      \abs{Y_I(\gamma(t))-Y_I(x)}
\le C\frac{t}{\eta}\abs{Y_I(x)} \quad\text{for all $t\in[0, C^{-1}]$.}  
\end{equation} 
\end{remark}

An immediate consequence of Theorem \ref{toros} is the following corollary.
\begin{corollary}\label{toro} 
 For all $x_0\in \R^n$, the number  $p_x:=\dim \Span\{Y_{1,x}\dots,
Y_{q,x}\}$
is constant if $x\in \cal{O}^{x_0}$.
\end{corollary}

To prove Theorem \ref{toros} and Corollary \ref{toro} we need the
following lemma.

\begin{lemma}\label{carica} 
      Let $p\le n$ and let  $U_1, \dots, U_p$  and $X =
\sum_{\a=1}^n f^\a \p_\a$ be smooth vector fields in $\R^n$.  
For any $K= (k_1, \dots, k_p)\in\I(p,n)$, we have
\begin{equation*}
\begin{aligned}
      X\big(dx^K(U_1, \dots, U_p)\big)& = \sum_{1\le \a\le p}
dx^K(U_1, \dots,U_{\a-1},[X, U_\a], U_{\a+1},\dots, U_p)
\\&\quad +\sum_{1\le \g\le n}\sum_{1\le\b\le p}
  \p_\gamma f^{k_\b}dx^{(k_1,\dots, k_{\b-1},\g,k_{\b+1},\dots, k_p)}(U_1,\dots, U_p).
\end{aligned}
\end{equation*}
\end{lemma}
\begin{proof}[Proof of Lemma \ref{carica}]
      See \cite{Street} or  \cite{MontanariMorbidelli11d}.
\end{proof}

\begin{proof}[Proof of Theorem \ref{toros}]
Fix two bounded  open sets  $\Omega_0\subset\subset\Omega_1\subset\R^n$
such that $\Omega\subset\subset\Omega_0$.
Let  $x\in\Omega$ and take a subunit path  $\gamma$ 
 such that $\gamma(0)= x$. Note that $\gamma $ is
the a.e.~solution of a  problem of the form 
$\dot\gamma = \sum_{j=1}^qu_j(t) Y_j(\gamma)$ where $\gamma(0)=x$ and
$u\in L^\infty((0,1), \R^q)$  satisfies $\abs{u(t)}\le 1$, a.e.  
We can approximate  $\gamma$ with a path
$\gamma^\s$ obtained as the solution of the problem  
\[
\dot\gamma^\s = \sum_{1\le j\le q}u_j(t) Y_j^\s(\gamma^\s)\quad\text{a.e., with}\quad \g^\s(0)=x.
\]
Since the vector fields
are locally Lipschitz, there is $C$ depending on $\Omega$,  $\Omega_0$
and on the Lipschitz constants of the vector fields such that
$\gamma^\s(t)\in \Omega_0 $ for all $t\in[0, C^{-1}]$.
Recall by
\eqref{ingrao} that for all $Z\in\pm\cal{P}$, we may write 
 $[Z^\s,Y_j^\s]= \sum (c_{j}^{i})^{(\s)}Y_i^\s + \s
b_j^\s\cdot\nabla$ for suitable   functions
$(c_{j}^{i})^{(\s)}$ and $b_j^\s$ smooth for all $\s>0$  and   bounded
uniformly
in $\s$. Namely, there are $C$ and $\wt \s$ depending on the choice of
$\Omega_0$ and $\Omega_1$ such that we have 
\begin{equation*}
      \sup_{\Omega_0}( \abs{(c_{j}^{i})^{(\s)}}+\abs{b_j^\s})\le C
\quad\text{for all $\sigma\le \wt \s$.}
\end{equation*}

 Fix $J\in \I(p,q) $ and $K\in \I(p,n)$.
Note that  each function $t\mapsto Y_h^\s(\g^\s(t))$ is Lipschitz continuous. 
At any differentiability point $t$ of $\gamma^\s$ we have $\frac{d}{dt}dx^K(Y_J(\gamma^\s_t))
= \sum_{h=1}^q u_h(t) Z_h^\s(dx^K(Y_J^\s))(\gamma_t^\s)$, where $\abs{u_j(t)}\le 1$.
By Lemma \ref{carica} we get
\begin{equation*}\begin{aligned}
 Z_h^\s (dx^K(Y_J^\s))(\g_t^\s)&
= \sum_{\a=1}^p \sum_{i=1}^q (c_{j_\a}^i)^{(\s)}(\g_t^\s)
dx^K (\dots, Y_{j_{\a-1}}^\s, Y_i^\s , Y_{j_{\a+1}}^\s , \dots,
Y_{j_{p}}^\s)(\g_t^\s)
\\& +\sum_{\a=1}^p \s dx^K
( \dots, Y_{j_{\a-1}}^\s(\g_t^\s), b_{j_\a}^\s(\g_t^\s)\cdot \nabla,
Y_{j_{\a+1}}^\s(\g_t^\s),\dots  )
\\&+\sum_{\g=1}^n \sum_{\b=1}^p \p_\g (g^{k_\b} )^{(\s)}
dx^{( k_1,\dots, k_{\b-1})}\wedge dx^\g \wedge
dx^{(k_{\b+1},\dots,,k_p)}(Y_J^\s)(\g_t^\s)
\\&=: A_1^\s + A_2^\s + A_3^\s.
 \end{aligned}
\end{equation*}
Denote $\Lambda_p^\s(y):= \big(dx^K Y_J^\s(y)\big)_{J\in\I(p,q), K\in \I(p,n)}$.
Since $ (c_{j_\a}^i)^{(\s)}$, $\p_\g g^{k,(\s)}$
and $ b_{j_\a}^\s$ are uniformly bounded as $\s$ tends to $0$,   we get
$ |A_1^\s|+ |A_2^\s|+|A_3^\s|\le C_1\abs{\Lambda_p^\s(\g_t^\s)}  + \s C_2$. This
gives for a.e.~$t$  
\begin{equation}
      \label{africa}
\Bigl|\frac{d}{dt} dx^K(Y_J^\s(\g_t^\s))\Bigr|\le
C_1\abs{\Lambda_p^\s(\gamma_t^\s)}+\s C_2 \quad\text{for all $J\in\I(p,q)$ $K\in\I(p,n)$.}
\end{equation} 
Therefore,
\begin{equation}\label{25} 
\begin{aligned}
 \Big|\frac{d}{dt} \Lambda_p^\s(\g_t^\s)\Big|&=
 \Big|\Big(\frac{d}{dt} \Lambda_p^\s(\g_t^\s)\Big)_{J\in\I(p,q), K\in\I(p,n)}\Big|
\\&\le  C_1 \abs{\Lambda_p^\s(\g_t^\s)} + C_2 \sigma.
\end{aligned}
\end{equation}
Integrating \eqref{25}  and  using the  Gronwall inequality\footnote{for all $a\ge 0$, $b>0$,
$T>0$ and $f$ continuous on $ [0,T]$,
\begin{equation}\label{grammatica}
 0\le f(t)\le a t+ b\int_0^t f(\t) d\t \quad\text{on
$t\in [0, T]$}\quad\Rightarrow
\quad f(t)\le \frac{a}{b} (e^{b t}-1)\quad\text{on  $t\in[0,T]$.}
\end{equation}} we obtain 
\[
\bigl|\Lambda_p^\s(\g_t^\s)  -
\Lambda_p^\s(x)\bigr|\le
  \Big(\abs{ \Lambda_p^\s(x)}+ \frac{C_2}{C_1}\s\Big) \big(e^{C_1
t}-1\big)
\]
for all $t\in [0, C^{-1}]$. Therefore,
as $\s\to 0$,  we get the conclusion
\begin{equation*}
\bigl| \Lambda_p (\g_t )  -
 \Lambda_p (x)\bigr|\le
 \abs{   \Lambda_p (x)}  \big(e^{C_1 t}-1\big)\quad\text{for all $t\in[0,C^{-1}]$.}
\end{equation*}
Estimate   \eqref{ippi} follows trivially.
The proof is concluded.
\end{proof}

\begin{proof}[Proof of Corollary \ref{toro}]
      Let $p\in\{1,\dots,\min\{q,n\}\}$. Let $I\subset\R$ be an interval.
Let $\gamma:I\to \R$ be a
subunit path. 
Let $A_p:=\{t\in
I:\abs{\Lambda_p(\gamma(t))}=0\}$. 
We claim that $A_p$ is open and closed. This will imply that either
$A_p=\varnothing$ or $A_p= I$. Then the proof is concluded.

To show the claim, note that the set is closed because it is the zero
set of the continuous function $I \ni t\mapsto
\abs{\Lambda_p(\gamma(t))}\in\R$.  
It is open as an easy consequence of estimate \eqref{stoic2}.
\end{proof}

\subsection{Manifold structure of orbits}\label{datt} 
In order to show our main theorem, we start with the following
commutativity lemma, which generalizes to the nonconstant rank the results in
\cite{RamSus,Rampazzo}.

Let  $\P=\{Y_1, \dots, Y_q\}$ be our family of vector fields and let  $\O= \O_\P^{x}$ be    a given
 orbit. We denote 
  $d$-balls by $ B_d(x,r)$ and Euclidean balls by $B(x,r)$.
Recall that the distance $d$ equips the orbit with a natural topology $\t_d$ which makes $\O$ connected.

\begin{lemma}
      Assume that the hypotheses of Theorem \ref{manno} hold. Let
$x_0\in\R^n$ and let    $p:= p_{x_0}$. Take also $I\in\I(p_{x_0},q)$ such that $\abs{Y_I(x)}\neq
0$. Then there exist $\e,\delta>0$, 
a neighborhood $U_{x_0}:= B_d(x_0, \delta)\subset B(x_0, \e)\cap \O$ of $x_0$ in
the orbit distance $d$ and a map $\b \in\Lip(B(x_0,\e),\R^{p\times
p})$ such that, letting 
\begin{equation}\label{vuj}
V_j :=\sum_{1\leq k\leq p} \b_j^k Y_{i_k}\quad\text{for all
$j=1,\dots,p$},
\end{equation} 
we have 
\begin{equation}\begin{aligned}
\label{caravan}
 \Span\{ V_{1,x}, \dots, V_{p,x}\} & = \Span\{ Y_{i_1,x}, \dots,
Y_{i_p,x}\} 
\\&
=
\Span\{Y_{1,x},\dots, Y_{p,x},\dots,Y_{q,x}\}
\end{aligned}
\end{equation}
for all $x\in
 U_{x_0}$. Furthermore, 
for all $j,k\in\{1,\dots,p\}$, for all $x\in U_{x_0}$ and for all small
$\abs{t_j},\abs{t_k}$, we have
the commutativity formula
\begin{equation}\label{commu} 
e^{-t_jV_j }e^{-t_k V_k } e^{t_j V_j} e^{t_k V_k} x=x.
\end{equation}  
\end{lemma}
\begin{proof}
Let $x_0\in\R^n$. Let $p=p_{x_0}$ and assume without loss of
generality that $(i_1, \dots, i_p)= (1,\dots,p)$, i.e.~that
$Y_1(x_0),
\dots, Y_p(x_0)$ are independent. Recall notation $Y_j:= \sum_{\a=1}^n g_j^\a
\p_\a$. Up to reordering coordinates, we may
as well assume that the matrix 
$(g_j^k(x_0))_{j,k=1,\dots,p}$ is nonsingular.

Define for $\s\ge 0 $ the functions  $\b_j^{k,\s}\in \Lip
(B(x_0, \e),
\R)$
such that  \begin{equation} 
\sum_{1\le k\le p} \b_i^{k,\s}(x)g_{k}^{\ell,\s}(x)= \d_i^\ell     
\end{equation} 
 for all
$i,\ell=1,\dots,p$. Here we let  $g_{k}^{\ell,\s}:=
(g_{k}^{\ell})^{(\s)}$. The functions  
$\b_i^{k,\s}$ are uniquely defined and, if $\e$ is sufficiently  small,
their Lipschitz constants on $B(x_0,\e)$
 are uniform, as $0\le \s\le \wt\s$ for some $\wt \s>
0$.

Define for all $x\in B(x_0, \e)$, 
$\ell=1,\dots,p$ and $0\le \s\le \wt \s$,
\begin{equation*}
\begin{aligned}
 V_{\ell,x}^\s&:  = \sum_{1\le k\le p} \b_\ell^{k,\s}(x) Y_{k,x}^\s
=: \p_\ell + \sum_{p+1\le i\le n}\phi_\ell^{i,\s}(x) \p_i,
\end{aligned}
\end{equation*}
where for $\ell\le p$ and $i\ge p+1$ we defined  $\phi_\ell^{i,\s} =
\sum_{k=1}^p\b_\ell^{k,\s} g_{k}^{i,\s}$. Note that   $\dim \Span\{Y_j^\s(x_0):1\le j\le q\}\ge
\dim \Span\{Y_j(x_0):1\le j\le q\} $ and inequality can be strict. 
Observe that, at the level $\s=0$ we have, by Corollary \ref{toro} and Remark \ref{pku},
\begin{equation}\label{goj} 
 \Span\{ V_{1,x}, \dots, V_{p,x}\} = \Span\{ Y_{1,x}, \dots,
Y_{p,x}\} =
\Span\{Y_{1,x},\dots, Y_{p,x},\dots,Y_{q,x}\}
\end{equation}
for all $x\in
 U_{x_0}$, where $U_{x_0}:= B_d(x_0, \delta)\subset\cal{O}\cap B(x_0,\e)$ is   a small
 $d$-ball centered at 
$x_0$.   Then \eqref{caravan} is proved. 
Note that the second equality of \eqref{goj} could  be false on any Euclidean
neighborhood of $x_0$. 
We can only claim that it holds if  $x$ belongs to a suitable 
$d$-neighborhood of $x_0$. Note that the topology defined by $d$ can be strictly stronger than the Euclidean one.
 This phenomenon makes inapplicable
the arguments of \cite[Theorem 5.3]{RamSus} to show the commutativity formula \eqref{commu}.

In order to show the commutativity formula
\eqref{commu},  first observe that 
\begin{equation}\label{ovv} 
 [V_j^\s, V_k^\s]_x \in\Span \{\p_{p+1}, \dots, \p_n\}\quad\text{for
all $j,k\in\{1,\dots p\}\quad x\in B(x_0, \e)$.}
\end{equation} 
 Let $x\in U_{x_0}= B_d(x_0, \delta)$. Take numbers $t,r$ with $\abs{t},\abs{r}$ small and 
start from the following formula, see \cite[Section 4]{RamSus}
\begin{equation}\label{lucida} 
\begin{aligned}
& e^{-tV_j^\s} e^{-r V_i^\s} e^{t V_j^\s} e^{r V_i^\s}  x-x
\\&= \int_0^td\t\int_0^r
d\r [V_i^\s,V_j^\s] (e^{-\t V_j^\s} e^{-\r V_i^\s}) 
(e^{(\r-r)V_i^\s} e^{\t
V_j^\s} e^{r V_i^\s}x)
\\&=
\int_0^td\t\int_0^r
d\r\Big\langle  [V_i^\s,V_j^\s] 
(e^{(\r-r)V_i^\s} e^{\t
V_j^\s} e^{r V_i^\s}x),\nabla(e^{-\t V_j^\s} e^{-\r V_i^\s}) 
(e^{(\r -r)V_i^\s} e^{\t
V_j^\s} e^{r V_i^\s}x)\Big\rangle.
\end{aligned}
\end{equation}
Note that the term with Euclidean gradient is uniformly bounded, i.e.
\begin{equation*}
\Bigl|\nabla(e^{-\t V_j^\s} e^{-\r V_i^\s}) 
(e^{(\r-r)V_i^\s} e^{\t
V_j^\s} e^{r V_i^\s}x) \Bigr|\le C \end{equation*}
as soon as $x$ is close to $x_0$ and 
the positive numbers $\abs{\r},\abs{\t}$  and $\sigma$  are sufficiently
small.

To prove   formula \eqref{commu}, it suffices to show that, if 
$d(x, x_0) $,    $\abs{t}$ and $\abs{r} $
are sufficiently small, then 
\begin{equation}\label{lobo} 
 [V_i^\s,V_j^\s] 
(e^{(\r-r)V_i^\s} e^{\t
V_j^\s} e^{r V_i^\s}x)\to 0\quad\text{ as $\s\to 0$,}
\end{equation} 
uniformly in the variables  $\r,\t$.
Write 
$
 \gamma^\s:= e^{(\r-r)V_i^\s} e^{\t
V_j^\s} e^{r V_i^\s}x$ and 
$\gamma
:= e^{(\r-r)V_i} e^{\t
V_j} e^{r V_i}x$.
Note that to show  \eqref{lobo}, we shall use the integrability
condition \eqref{ingrao}.

Write first, 
 by \eqref{ingrao}, at any point of the Euclidean
ball  $B(x_0, \e)$ 
\begin{equation}\label{giochi} 
\begin{aligned}
 & [V_i^\s, V_j^\s] =\sum_{h,k=1}^p [\b_i^{k,\s}Y_k^\s,
\b_j^{h,\s}Y_h^\s]
\\&=\sum_{h,k\le p}
(\b_i^{k,\s}Y_k^\s \b_j^{h,\s} -  \b_j^{k,\s}Y_k^\s \b_i^{h,\s})Y_h^\s
 +\sum_{ h,k\le p} \b_i^{k,\s}
\b_j^{h,\s} \Big\{\sum_{1\le s\le q}(c_{kh}^s)^{(\s)} Y_s^\s + \s
b_{kh}^\s\cdot\nabla\Big\} 
\\&= 
\sum_{h,k,\nu\le p}
(\b_i^{k,\s}Y_k^\s \b_j^{h,\s} -  \b_j^{k,\s}Y_k^\s
\b_i^{h,\s})g_h^{\nu,\s}V_\nu^\s
 +
\sum_{h,k,s,\nu=1}^p \b_i^{k,\s}
\b_j^{h,\s}  c_{hk}^{s,(\s)} g_s^{\nu,\s}V_\nu^\s 
\\&\qquad  + \sum_{h,k\le p} \sum_{s\ge p+1}
\b_i^{k,\s}
\b_j^{h,\s} c_{hk}^{s,(\s)} Y_s^\s 
 +
\s\sum_{h,k\le p}  
\b_i^{k,\s}
\b_j^{h,\s}
b_{kh}^\s\cdot\nabla
\\&  
 =: \sum_{\nu\le
p}A_{ij}^{\n,\s}V_\n^{\s} +\sum_{s\ge
p+1}B_{ij}^{s,\s}Y_s^\s+\s C_{ij}^\s\cdot\nabla,
\end{aligned}
\end{equation}
We must write this equality  at the point $\g^\s$ 
and then let $\s\to
0$. Note that the terms
$Y_k^\s \b_i^{h,\s}$ and  $c_{h,k}^{s,(\s)}$ are mollifiers of  
$L^\infty_\loc $ functions. Then we can not establish at this stage
their behavior as  $\s\to 0$.
However by \eqref{ovv}, we know that the projection of the left hand-side along 
$\p_\ell$ vanishes  for all  
$\ell=1,\dots,p$.

We claim now that  \eqref{giochi} can be
 in fact written at the point $\g^\s$ in the form 
\begin{equation}\label{star} 
 [V_i^\s,V_j^\s]_{\g^\s}= \sum_{k\le p}\s
M_{ij}^{k,\s}V_{k,\g}+\s N_{ij}^\s\cdot\nabla,
\end{equation}
where both $M_{ij}^{k,\s}$ ans $N_{ij}^\s$ are  bounded uniformly, as $\s\to
0$.\footnote{In \eqref{star} we 
identify tangent spaces  to $\R^n$ at different points.}
This shows 
\eqref{lobo}. Then \eqref{commu} follows and the proof of the lemma is concluded.

To prove the  claim \eqref{star}, observe first that  
since the vector fields $Y_j= g_j\cdot\nabla$ 
are locally Lipschitz, we know that $
\abs{\g^\s-\g}\le C\s $. Therefore, we have 
\begin{equation}\begin{aligned}
 \label{aa} \abs{g_s^\s(\g^\s)-g_s(\gamma)} & \le C\s\quad\text{for $1\le s\le q$,}
\\  \abs{\b_i^{k,\s}(\g^\s)-\b_i^k(\g)} & \le C\s\quad\text{for
$i,k\in\{1,\dots,p\}$,}
\end{aligned}\end{equation} 
where we used the fact that the functions  $\b_i^k$ are  Lipschitz in some
neighborhood of $x_0$. Moreover, since $d(\gamma,x_0)\le \delta+C(\abs{r}+\abs{t})$ is small, 
  in view of \eqref{goj}
we can write for  $\s=0$ by the Cramer's rule \eqref{cromo}
\begin{equation}\label{cromi} 
 Y_{s,\g} = \sum_{\nu=1}^p\frac{\langle Y_{I,\g}, \iota^\nu(Y_{s,\g})
Y_{I, \g} \rangle}{\abs{Y_{I,\g}}^2} Y_{\nu,\gamma}\quad\text{for all $s\in\{1,\dots,q\}$}
\end{equation} 
 (formula \eqref{cromi} is nontrivial  only if $s\ge p+1$). The ratio is bounded  by Remark~\ref{pku}. 
Note also that calculating \eqref{giochi} at the point $\gamma^\s$, we get
\begin{equation}\label{giam} 
 [V_i^\s, V_j^\s]_{\gamma^\s}= \sum_{k\le
p}A_{ij}^{k,\s}(\g^\s)V_{k,\g^\s}^{\s} +\sum_{\ell\ge
p+1}B_{ij}^{\ell,\s}(\g^\s)Y_{\ell,\g^\s}^\s+\s
C_{ij}^\s(\g^\s)\cdot\nabla,
\end{equation}
where
$\abs{A_{ij}^{k,\s}(\g^\s)},\abs{ B_{ij}^{\ell,\s}(\g^\s)},\abs{C_{ij}^\s (\g^\s)}\le C$ for
some $C$  depending on $I$, $x_0$,  $x$, $\abs{t}$ and $\abs{r}$, but uniformly bounded as
$\sigma\to 0$.

Next write for $k\le p$,
$
 V_{k,\g^\s}^\s = V_{k,\gamma}+\s D_k^\s\cdot\nabla
$, where $D_k^\s$ are uniformly bounded by \eqref{aa}. Moreover,  
for $\ell\ge p+1$, using \eqref{aa} and \eqref{cromi} write
\begin{align*}
 Y_{\ell,\g^\s}^\s& =: Y_{\ell,\g} +\s E_\ell^\s \cdot\nabla
=\sum_{\nu\le p}\frac{\langle Y_{I,\g},
\iota^\nu(Y_{\ell,\g})Y_{I,\g}\rangle}{\abs{Y_{I,\g}}^2}
Y_{\nu,\g} + \s E_\ell^\s\cdot\nabla 
\\&
=\sum_{\nu,k\le p}\frac{\langle Y_{I,\g},
\iota^\nu(Y_{\ell,\g})Y_{I,\g}\rangle}{\abs{Y_{I,\g}}^2}
g_{\nu}^k(\g) V_{k,\g} + \s E_\ell^\s\cdot\nabla .
\end{align*}
Here $E_\ell^\s$ is bounded by \eqref{aa}.
Inserting into \eqref{giam} this gives 
\begin{align*}
[V_i^\s, V_j^\s]_{\g^\s} &=\sum_{k\le p}A_{ij}^{k,\s} (\g^\s)(V_{k,\g}+ \s D_k^\s\cdot\nabla)
\\& \quad +      
\sum_{\ell\ge p+1}B_{ij}^{\ell,\s} (\g^\s)\Bigl(\sum_{\nu,k\le p}
\frac{\langle Y_{I,\g}, \iota^\nu(Y_{\ell,\g})Y_{I,\g}\rangle}{\abs{Y_{I,\g}}^2} g_\nu^k(\g) V_{k,\g}
+\s E_\ell^\s\cdot\nabla\Bigr)
\\&\quad + \sigma C_{ij}^\s(\g^\s)\cdot\nabla.
\end{align*}
Projecting along $\p_k$ with $k\le p$ we get, in view of \eqref{ovv}
\begin{equation*}
 A_{ij}^{k,\s}(\g^\s)+\sum_{\ell\ge p+1}\sum_{\nu\le p}B_{ij}^{\ell,\s}(\g^\s)
\frac{\langle Y_{I,\g},
\iota^\nu(Y_{\ell,\g})Y_{I,\g}\rangle}{\abs{Y_{I,\g}}^2}
g_{\nu}^k(\g)=O(\s)=:\sigma M_{ij}^{k,\sigma}\quad\text{for all $k\le p$}.
\end{equation*}
The functions   $M_{ij}^{k,\s}$ are bounded. 
 Therefore  the claim \eqref{star} is proved and the proof of the lemma is concluded.
\end{proof}

Now we are ready to show that orbits are submanifolds. The argument is known, see \cite{Hermann77} 
for the smooth case, 
\cite[Section 6.0.2]{Rampazzo} for the nonsmooth constant-rank case.

\begin{proof}[Proof of Theorem \ref{manno}]
   Let $x_0\in \R^n$. Let $p:=p_{x_0}$, take $I\in\I(p, q)$ 
such that $\abs{Y_I(x_0)}\neq 0$ and   define   for some small  $\delta>0$
 and $u\in
B(0,\delta) \subset \R^p$
the exponential map
\begin{equation}\label{gimpcall} 
\Phi (u) : =\exp\Big(\sum_{j=1}^p u_j V_j \Big) x_0,  
 \end{equation}  where the vector fields
$V_j$ are defined in \eqref{vuj}.
Observe first that $\Phi(u)\in\O$, for 
 all $u\in B(0,\delta)$, where $\delta$ is sufficiently small.   Indeed,  $\Phi(u)$ is solution of the ODE
\begin{equation*}
      \dot\gamma= \sum_{j=1}^p  
u_j V_j(\gamma(t))= \sum_{j,k=1}^p u_j \beta_j^k(\gamma(t))Y_{k}(\gamma(t))\quad
\text{with $\gamma(0)= x_0$.}
\end{equation*}
Therefore, since the coeffients  $\beta_j^k$ are bounded in a neighborhood of $x$, $\gamma$ is subunit (possibly up to a linear  reparametrization).
Note the obvious  inclusion
\begin{equation}
      \label{trovo}
\Phi(B(0, \d))\subset B_d(x_0, C\d )\quad\text{for all small $\delta>0$.} 
\end{equation}

Next, by the commutativity property established in \eqref{commu} 
we may claim that for all $k\in\{1,\dots,p\}$
\begin{equation*}
 \frac{\p}{\p u_k} \Phi (u) = \frac{\p}{\p u_k}e^{u_k
V_k}\exp\Big(\sum_{j\neq k}
u_j V_j \Big) x_0 =
V_{k,\Phi(u)}.
\end{equation*}
Therefore $\Phi\in C^{1,1}(B(0,\delta), \R^n)$, if
the positive
number $\delta$ is sufficiently small.  Moreover, possibly shrinking 
$\delta$,
the set $\Phi(B (0,\delta))\subset\cal{O}$ is a $p$-dimensional
$C^{1,1}$ submanifold
embedded in $\R^n$ with $T_x \Phi (B(0,\delta)) = P_x$ for all $x\in
\Phi(B(0,\delta))$.

Let $\Sigma:=\Phi(B(0,\delta))$ with $\delta >0$ sufficiently small.
 We claim that 
for all $y\in \Sigma$ there is $\sigma>0$ such that \begin{equation}\label{sottobosco} 
B_d(y, \sigma)\subset \Phi(B(0,\delta))                 
\end{equation} 

To prove the claim, let $z\in B_d(y,\sigma)$. This means that $z=\gamma(1)$, where $\gamma\in \Lip((0,1),
 \R^n)$ satisfies a.e.~$\dot\gamma= \sum_j c_jY_j(\gamma)$ with $\abs{c(t)}\le \sigma$ and $\gamma(0)=y$.
We may assume that there is a small Euclidean neighborhood $U$ of $y$, a neighborhood $V$ of the origin in $\R^n$ and a $C^{1,1}$ diffeomorphism  $F:U\to V$ such that $F(y)=0$ and $F(\Sigma\cap U)= V\cap\{(x', x'')\in\R^p\times\R^{n-p}: x''=0\}.$ Choose $\sigma$ small enough to ensure that $\gamma(1)\in U$ for all $c\in L^\infty(0, 1)$ with $\abs{c(t)}\le \sigma$ a.e.
Since $F\in C^{1,1}$, the Lipschitz path $\eta(t):= F(\gamma(t))\in V$ satisfies for a.e.~$t\in[0,1]$
\begin{equation}\label{trallo} 
      \dot\eta(t)
 = dF(\gamma(t))\dot\gamma(t) = \sum_j c_j(t)Y_jF(\gamma(t))=:\sum_j c_j(t)(F_*Y_j)(\eta(t)),
\end{equation}
with $\eta(0)= 0$. 
This
 Cauchy problem has a
 unique solution, because $F_*Y_j$ is Lipschitz.
 Moreover since $Y_j(z)\in T_z\Sigma$ for all $z\in\Sigma$, it must be 
$F_*Y_j= \sum_{\a=1}^n h_j^\a(\xi)\p_{\xi_\a}$, where $h_j^\a(\xi', 0)=0$
for all $(\xi',0) \in V\cap (\R^p\times \{0\})$, $\a\ge p+1$ and $1\le j\le p$.
 Since the coefficients $h_j^\a$
are Lipschitz continuous in $V$, we conclude that it must be $\eta''(t)=0$ for $t\in[0,1]$. 
This ends the proof of inclusion \eqref{sottobosco}.

Next observe that we can repeat the construction of the map $\Phi$ in \eqref{gimpcall}  at any point $x_0\in\O$ and for any
possible choice of $I\in\I(p_{x_0}, q)$ such that $\abs{Y_I(x_0)}\neq
0$.

As a first consequence  of inclusion \eqref{sottobosco},  we show that the family of  sets of the  form $\Phi(B_\delta)$  constructed in this way can be used as a basis for a topology  $\t_\O$ on $\O$.  To show this,
let $x, \wt x\in\O$. Let
\[
      \Phi_x(u)= \exp\Big(\sum_j u_j V_j\Big)x
\quad\text{and}\quad
 \wt\Phi_{\wt x}(u)= \exp\Big(\sum_j u_j \wt V_j\Big)\wt x \] be the maps constructed as above.
  Let $\delta$  and $\wt\delta$ be sufficiently small to ensure that 
 $\Sigma:= \Phi_x(B_\d)$ and 
$ \wt \Sigma:= \wt\Phi_{\wt x}(B_{\wt\d})$
 are manifolds. Assume also that $\Sigma\cap\wt \Sigma\neq\varnothing$. Let now 
 $x^*\in\Sigma\cap\wt\Sigma$  and choose a map 
$\Phi^*_{x^*}(u)= \exp(\sum_ju_j V^*_j )(x^*).
$
We need to prove that, for sufficiently small $\delta^*>0$,  we have
  $\Phi^*_{x^*}(u)\in\Sigma\cap\wt\Sigma$
 for each $u\in B(0,\delta^*)$. 
But we  proved before that $\Phi_{x^*} (B(0, \delta^*))\subset B_d(x^*, C\delta^*)$, if $\delta^*$
is sufficiently smooth. Therefore the claim follows from the already proved inclusion \eqref{sottobosco}.

A further consequence of inclusions \eqref{trovo} and  \eqref{sottobosco}  is that the topologies $\t_\O$ and $\t_d$ are equivalent. They are both stronger that the Euclidean topology  $\t_\Eucl$  restricted to $\O$. Moreover, on a small piece of leaf   of the form $\Sigma=\Phi(B(0,\delta))$,  $\t_\Eucl,\t_\O$ and $\t_d$ induce
all the same topology.

Finally, the family of maps in \eqref{gimpcall} as $x_0$ varies in a fixed  orbit $\O$,  $I\in\I(p_{x_0}, q)$
is such that $\abs{Y_I(x_0)}\neq 0$ and $\delta$ is sufficiently small,  can be used   to give a structure of differentiable manifold to $(\O, \t_d)$. 
\end{proof}

\footnotesize

\def\cprime{$'$}
\providecommand{\bysame}{\leavevmode\hbox to3em{\hrulefill}\thinspace}
\providecommand{\MR}{\relax\ifhmode\unskip\space\fi MR }
\providecommand{\MRhref}[2]{%
  \href{http://www.ams.org/mathscinet-getitem?mr=#1}{#2}
}
\providecommand{\href}[2]{#2}

\normalsize
\bigskip \noindent\sc \small  Annamaria Montanari, Daniele Morbidelli
\\ Dipartimento di Matematica,
Universit\`{a} di Bologna  (Italy)
\\Email: \tt   annamaria.montanari@unibo.it,
daniele.morbidelli@unibo.it

\end{document}